\documentclass[a4paper,12pt]{article}
\usepackage[utf8]{inputenc}
\usepackage[T1]{fontenc}
\usepackage{upgreek}
\usepackage{amsmath,amssymb,amsfonts,amsthm,amstext}
\usepackage{dsfont}
\usepackage[scr]{rsfso}
\usepackage{nicefrac}
\usepackage{bm}
\usepackage{enumitem}
\usepackage{titling}
\usepackage{cite}

\numberwithin{equation}{section}

\newtheorem{Thm}{Theorem}
\newtheorem{Pro}{Proposition}[section]
\newtheorem{Lem}{Lemma}[section]
\newtheorem{Rmk}{Remark}[section]

\newtheorem{Cor}{Corollary}[section]

\newcommand{\D}{\Delta}
\newcommand{\dg}{\dagger}
\newcommand{\E}{\mathds{E}}

\newcommand{\e}{\eta}
\newcommand{\et}{\eta_t}
\newcommand{\ets}{\eta_t^{e_o}}
\newcommand{\eps}{\varepsilon}

\newcommand{\Fc}{\mathcal{F}}
\newcommand{\g}{\gamma}

\renewcommand{\l}{\mathscr{L}}
\newcommand{\lbd}{\lambda}

\newcommand{\Nc}{\mathcal{N}}
\renewcommand{\O}{\Omega}

\newcommand{\p}{\mathds{P}}

\renewcommand{\r}{\rho}

\renewcommand{\S}{\mathcal{S}}

\newcommand{\T}{\mathds{T}}

\newcommand{\xis}{\xi_t^{e_o}}
\newcommand{\xit}{\xi_t}

\newcommand{\wij}{w_{i\to j}}

\newcommand{\Z}{\mathds{Z}}
\newcommand{\1}{\mathds{1}}

\usepackage{lmodern}
\usepackage{titling}

\title{\textsc{\Large Multiple Phase Transitions for an Infinite System of Spiking Neurons}}

\author{\textsc{\large A. M. B. Nascimento}\thanks{Partially supported by CNPq grant 155972/2018-9}\thanksgap{1ex}\thanks{Centro de Ciências Exatas e da Terra, Universidade Federal do Rio Grande do Norte, Av. Senador Salgado Filho 3000, Campus Universitário de Natal, 59078-970 Natal RN, Brasil. (email: antonio.nascimento@ufrn.br)}}

\date{}

\begin{document}

\maketitle

\vspace{-1cm}

\begin{abstract}
We consider a stochastic model describing the spiking activity of a countable set of neurons spatially organized into a homogeneous tree of degree $d$, $d \geq 2$; the degree of a neuron is just the number of connections it has. Roughly, the model is as follows. Each neuron is represented by its membrane potential, which takes non-negative integer values. Neurons spike at Poisson rate 1, provided they have strictly positive membrane potential. When a spike occurs, the potential of the spiking neuron changes to 0, and all neurons connected to it receive a positive amount of potential. Moreover, between successive spikes and without receiving any spiking inputs from other neurons, each neuron's potential behaves independently as a pure death process with death rate $\gamma \geq 0$. In this article, we show that if the number $d$ of connections is large enough, then the process exhibits at least two phase transitions depending on the choice of rate $\gamma$: For large values of $\gamma$, the neural spiking activity almost surely goes extinct; For small values of $\gamma$, a fixed neuron spikes infinitely many times with a positive probability, and for ``intermediate'' values of $\gamma$, the system has a positive probability of always presenting spiking activity, but, individually, each neuron eventually stops spiking and remains at rest forever.
\end{abstract}

\vspace{.5cm}
 \noindent{\bf AMS 2010 Mathematics Subject Classification.} 60G55, 60K35, 92B99
\vspace{.5cm}

 \noindent{\bf Key words and phrases:} system of spiking neurons, multiple phase transition, trees

\vspace{.5cm}

\section{Introduction}\label{SEC1}
In this article, we consider  a continuous-time Markov process describing the spiking activity propagation within an infinite network of neurons. This process, which we shall denote by $\xit$, is informally defined as follows. Let $V$ be an arbitrary countable set and consider a graph $G=(V,E)$ to represent the network of neurons; vertices in $V$ represent neurons, and edges in $E$ indicate the existence of a connection (or interaction) between them. For each $t\geq 0$, the random variable $\xit$, which takes values in $\Z_+^V$, where $\Z_+=\{0,1,\ldots\}$, is a configuration giving the membrane potential for all neurons in $V$, that is to say, $\xit(i)\in \Z_+$ stands for the membrane potential of neuron $i$ at time $t$. This configuration evolves in time, with random transitions governed by the following rules. To each neuron, we assign a Poisson process of rate 1 in such a way that when the exponential time occurs at an {\it active neuron} --- one with strictly positive membrane potential ---, it spikes. As a result, two things happen at the same time within the system: (i) the potential of the spiking neuron is immediately changed to a resting value, which is assumed to be 0, and (ii) all neurons in the network which are connected to the spiking neuron increase their potentials by a positive value, which may vary for each pair of neurons. Furthermore, between two consecutive spikes and without receiving any spiking inputs from other neurons, each neuron's potential behaves independently as a pure death process with death rate $\g\geq 0$ per unit of potential. This part of the dynamics is to model the so-called {\it leakage effect}, a natural phenomenon that causes neurons to spontaneously lose potential. The rate $\g$ will be, therefore, called the {\it leakage parameter}. In this article, we shall assume that neurons are spatially organized into infinite trees, which are, roughly speaking, connected graphs with undirected edges and no cycles. A rigorous description of this model will be given in the next section.

Since the neuron's potential vanishes each time it spikes, this process can be seen as an extension to the continuous-time framework of the model introduced by Galves and Löcherbach~\cite{Galves2013} (see also~\cite{Galves2015} for a critical review), where neurons are represented through their spike trains. There are many other versions, in discrete and continuous time, of this model in the literature; let us mention~\cite{Brochini2016,DeMasi2015,Duarte2014,Duarte2015,Fournier2016,Yaginuma2016,Ferrari2018}. We will also be interested in a particular version of the variant introduced by Ferrari et al.~\cite{Ferrari2018}, which is equivalent to a Markov process similar to ours, with the same spiking and leakage parameters, but for which we do not allow neurons to accumulate more than one unit of potential.

Our motivation for this project comes, in part, from that work by Ferrari et al.~\cite{Ferrari2018}. The authors have proved the existence of a phase transition situation for the neural spiking activity in that paper, which was interpreted as a sudden change of behavior in the number of spikes emitted by each neuron as a function of the leakage parameter. Specifically, by taking the network of neurons as the one-dimensional integer lattice $\Z$ (equipped with its usual graph structure), they found a distinguished value, say $\g_c$, $0< \g_c< \infty$, for the leakage parameter for which the following holds: If the process starts from a potential configuration in which all neurons are active, then there exists a positive probability that each neuron spikes infinitely many times for $\g<\g_c$, while that for $\g>\g_c$, each neuron eventually stops spiking and remains at rest forever. These behaviors are usually called {\it local survival} and {\it local extinction}, respectively, in the statistical physics literature, and so $\g_c$ can be seen as the critical point corresponding to local survival of the spiking activity in $\Z$.

From this scenario, a natural question emerges when the concept of {\it global survival} comes into play. By global survival, we mean that the neural spiking activity has a positive probability of always being present within the system. Let $\g_*$ be the critical value of $\g$, so that the neural spiking activity goes extinct with probability 1 if $\g$ is above $\g_*$, and survives globally if $\g$ is below $\g_*$. Local survival implies global survival; therefore $\g_c\leq \g_*$. We wonder if $\g_c<\g_*$. When this strict inequality holds, the process is said to exhibit two phase transitions.

The purpose of this article is to show that we have two phase transitions for the presented process when it starts from a state that has a single active neuron. Here, we shall only consider homogeneous trees to represent the network of neurons. A homogeneous tree is the unique connected tree in which every vertex has degree $d$, $d\geq 2$, where the degree of a vertex is the number of connections it has. On this structure, we show that the process exhibits two phase transitions, provided the number $d$ of connections is large enough, and at least one phase transition for all $d\geq 2$. In addition, we present some estimates for the expectation of the total membrane potential within the system at any instant of time.

Before finishing this introduction, let us briefly point out some of our motivations for taking neurons spatially organized into trees. Firstly, some theoretical works on systems of spiking neurons have considered the underlying graph as a slightly supercritical Erdös-Rényi random graph (see e.g.~\cite{Galves2013}), commonly based on arguments by Beggs and Plenz~\cite{Beggs2003}, who say that networks of living neurons should behave in a slightly supercritical state. However, it has been shown in~\cite{Jara2014} that the neighborhood of random points in a slightly supercritical Erdös-Rényi graph looks like the one of a fixed vertex in a tree. Therefore, trees seem to be as likely as Erdös-Rényi random graphs to represent, at least locally, real networks of neurons. Another reason is that the main purpose of the current article concerns the existence of multiple phase transitions for the presented process and, as suggested by Pemantle's work~\cite{Pemantle1992} on the contact process, trees are probably the most straightforward graph structure on which we can observe such behavior.

The remainder of the article is organized as follows. In Section~\ref{SEC2} below, we present the general definition of the model, introduce some notation, and state our main result, Theorem~\ref{Thm1}, whose proof follows easily from Theorem~\ref{Thm2} on bounds for the critical values corresponding to local and global survival. In Section~\ref{SEC3}, we prove Theorem~\ref{Thm2} and discuss some implications of its proof for a special version of the model introduced by Ferrari et al.~\cite{Ferrari2018}.

\section{Model definition and Main Results}\label{SEC2}
In this section, we present a general definition of the process we study. To do so, the following elements are needed:
\begin{enumerate}[wide,label={(\roman*)}]
	\item A graph $G=(V,E)$, consisting of a countable set of vertices $V$ and a set of (undirected) edges $E$, $E\subset V^2$, to represent the network of neurons. Graph $G$ is assumed to be connected and with vertices (neurons, from now on) of bounded degree. 
	
	\item A matrix $w=(\wij)_{i,j\in V}$ with non-negative integer entries to measure the potential gain induced by spikes; $\wij$ will be the value added to the potential of neuron $j$ when neuron $i$ spikes. We assume here that, for each $i\in V$, $\wij=0$ unless $j\in \Nc_i$, where $\Nc_i=\{j\in V: \{i,j\}\in E\}$ is the {\it interaction neighborhood} of $i$ in $G$. Moreover, we impose that
	\begin{align}
	\lbd:=\sup_{i\in V} \sum_{j\in \Nc_i} \wij <\infty.
	\end{align}
	
	\item A parameter $\g\geq 0$ to model the leakage effect; we will call $\g$ the {\it leakage parameter}.
\end{enumerate}

For a fixed $w$, the model we consider is thus defined as the continuous-time Markov process $\xit=(\xit(i), i\in V)$ on $G$ with state space being
\begin{align*}
\S=\left\{\xi\in\Z_+^V:\textstyle\sum_{i\in V}\xi(i)<\infty \right\},
\end{align*}
($\xi$ is a configuration giving the state of the membrane potential for the set $V$ of all neurons), whose generator is given, for any cylindrical function, by
\begin{align}\label{G}
L f(\xi)=\sum_{i\in V} \1\{\xi(i)\geq 1\}\left(f(\D_i^* \xi)-f(\xi)\right)+\g\sum_{i\in V}\xi(i) \left(f(\D_i^{\dg} \xi)-f(\xi) \right),
\end{align}
where for each $\xi\in \S$ such that $\xi(i)\geq 1$, $\D_i^* \xi$ is given by
\begin{align}\label{Map1}
\begin{aligned}
(\D_i^*\xi)(j)&=\xi(j)+\wij,\quad&&\text{if } j\in \Nc_i,\\
&=\xi(j),\quad&&\text{if }j\in V\setminus \left(\Nc_i\cup \{i\}\right),\\
&=0,\quad&&\text{if }j=i,
\end{aligned}
\end{align}
and for each $\xi\in \S$, $\D_i^{\dg} \xi$ is defined as
\begin{align}\label{Map2}
\begin{aligned}
(\D_i^{\dg} \xi)(j)&=\xi(j)-1,\quad&&\text{if }j=i,\\
&=\xi(j),\quad&&\text{if }j\neq i.
\end{aligned}
\end{align}

An interpretation of~\eqref{G} is as follows. Its first term describes the neural spiking activity within the system: If neuron $i$ is active, when it has a strictly positive membrane potential, then at rate 1 it spikes, instant at which its potential is reset to 0 and, at the same time, every neuron $j$ belonging to its interaction neighborhood $\Nc_i$ adds the amount of $\wij$ to its potential. The second term models the spontaneous loss of potential that neurons are subject to within the system, due to the leakage effect: If neuron $i$ has membrane potential equal to $k$, $k\geq 1$, then after an exponential time of rate $\g k$, its potential decreases to $k-1$.

For coupling purposes, it is also convenient to construct $\xit$ using a graph technique called graphical representation~\cite{Harris1978}. For that sake, we associate to each neuron $i\in V$ independent Poisson processes $N_i^*$ of rate 1, and $N_{i,k}^{\dagger}$ of rate $\g$ for $k\geq 1$. The $N_i^*$ process will give the times at which neuron $i$ spikes, and the $N_{i,k}^{\dagger}$ processes will give the times at which neuron $i$ is affected by the leakage. The construction of the process is done as follows. Starting with finitely many neurons with strictly positive potential, at an arrival time of $N_i^*$, if $\xit(i) \geq 1$, we replace its value by $0$ and the value of $\xit(j)$ by $\xit(j)+\wij$ for all $j\in \Nc_i$, and at an arrival time of $N_{i,k}^{\dagger}$, we replace the value of $\xit(i)$ by $\xit(i)-1$ if $k \leq \xit(i)$.

Along this paper we focus on the case $G=\T_d$, where $\T_d$ is the homogeneous tree with degree $d$, $d\geq 2$ (note that $\T_2=\Z$, the one-dimensional integer lattice). In doing so, it will be convenient to think of $\T_d$ arranged into {\it levels} in such a way that some distinguished neuron, denoted by $o$ and called the {\it root}, is in level $0$ and every neuron in level $n\in \Z$ has exactly one neighbor (its parent) in level $n-1$ and the other neighbors (its children) in level $n+1$. We shall use $l(i)$ to indicate the level of neuron $i\in V$. For future reference, let us set
\begin{align}\label{Dfabt}
\begin{aligned}
w_1& = \sup_{i\in V}\left\{\wij: l(j)=l(i)-1\right\},\\
w_2 &= \sup_{i\in V}\left\{\wij: l(j)=l(i)+1\right\}.
\end{aligned}
\end{align}
Clearly, all these quantities are finite.

We shall denote by $\xit^\xi$ the process starting from a potential configuration $\xi\in \S$ at time $0$. When $\xi$ equals $e_o$, 
where
\begin{align}
\begin{aligned}
e_o(i)&=1,\quad&&\text{if }i= o;\\
&=0,\quad&&\text{if }i\neq o,
\end{aligned}
\end{align}
we shall write $\xis$ 
(or $\xit^{e_o,\g}$ when it is necessary to make explicit the leakage parameter in the notation). 

Let $|\xis|=\sum_{i\in V} \xis(i)$ be the total membrane potential present within the system at time $t$. We shall say that $\xis$ {\it survives (globally)} if
\begin{align*}
\p_w\left(|\xis|\geq 1,\; \forall t\geq 0\right)>0
\end{align*}
where $\p_w$ stands for the law of the process with a fixed matrix $w$. When $\xis$ does not survive, we say it {\it dies out}. If
\begin{align*}
\p_w\left( \xis(o)\geq 1 \mbox{ i.o.}\right)>0
\end{align*}
we shall say that $\xis$ {\it survives locally}. 
We remark that, since $\T_d$ is a connected graph, the definition of local survival does not depend on the particular neuron $o$. We also remark that, given Lemma~\ref{Lem5} in Appendix~\ref{AP1}, the probability of global survival $\p_w(|\xit^{e_o,\g}|\geq 1,\; \forall t\geq 0 )$ is a non-increasing function of parameter $\g$.

Now, given $w$, let $\g_l$ and $\g_g$ be the following critical values for the leakage parameter $\g$:
\begin{align}
\begin{aligned}
\g_l&=\sup\left\{\g: \p_w\left( \xit^{e_o,\g}(o)\geq 1 \mbox{ i.o.}\right)>0\right\},\\
\g_g&=\sup\left\{\g: \p_w\left(|\xit^{e_o,\g}|\geq 1,\; \forall t\geq 0\right)>0\right\}.
\end{aligned}
\end{align}
In words, $\g_l$ is the critical value corresponding to the local survival and $\g_g$ is the critical value corresponding to the global survival of the spiking activity in $\T_d$. Clearly, $\g_l\leq \g_g$. Our main result, which is stated next, presents sufficient conditions to have $\g_l < \g_g$, and hence, two phase transitions for the process.

\begin{Thm}\label{Thm1}
	We have $\g_l<\g_g$ (two phase transitions for $\xis$) whenever $d> 4w_1w_2+1$.
\end{Thm}

As a particular conclusion of this result, we have the following.

\begin{Cor}
	If the matrix $w$ is such that for each $i\in V$, $\wij=\1\{j\in\Nc_i\}$, then $\g_l<\g_g$ for all $d\geq 6$.
\end{Cor}

Theorem~\ref{Thm1} follows immediately from Theorem~\ref{Thm2} below. Our strategy to proving it is standard, and consists in deducing an upper bound for $\g_l$ and a lower bound for $\g_g$ which implies that $\g_l < \g_g$ if $d$ is sufficiently large.

\begin{Thm}\label{Thm2}
	We have the following bounds for the two critical values:
	\begin{align}\label{ub1}
	0<\g_l\leq 2\sqrt{w_1 w_2(d-1)}-1
	\end{align}
	and
	\begin{align}\label{ub2}
	d-2\leq \g_g \leq \lbd-1
	\end{align}
\end{Thm}

Another immediate conclusion from Theorem~\ref{Thm2} follows. Since $\g_l\leq \g_g$, the bounds in~\eqref{ub2} and~\eqref{ub1}, combined with the fact that $\lbd \geq 2$, show that
\begin{Cor}
    For all $d\geq 2$, we have
    \begin{align*}
        0<\g_g \leq \lbd-1.
    \end{align*}
    Therefore, the process $\xis$ has at least one phase transition.
\end{Cor}

The proof of Theorem~\ref{Thm2} is done in Section~\ref{SEC3}. The strategy is as follows. The upper bounds in~\eqref{ub1} and~\eqref{ub2} are both obtained based on a method employed by Liggett~\cite{liggett2013stochastic} (see Theorem 4.8. therein) to compute the exact values of the two critical points for the branching random walk on homogeneous trees. Generally speaking, the method consists of relating global and local survival of the process to properties of an auxiliary process specified by a convenient weighting function defined on $\S$. To get the lower bound in~\eqref{ub2}, we apply a coupling argument that compares our process to a non-spatial branching process, which is constructed by using the graphical representation of the process and exploring the nice structure (no cycles) of trees.

\section{Bounds on the Critical Values --- Proof of Theorem~\ref{Thm2}}\label{SEC3}

Let $\nu_\r:\S\to [0,\infty)$ be a weighting function defined by
\begin{align}\label{wf}
\nu_\r(\xi)=\sum_{i\in V}\xi(i)\r^{l(i)},
\end{align}
where $\r$ is a non-negative parameter to be specified later. Note that $\nu_1(\xi)=|\xi|$.

In what follows, we shall consider and study the auxiliary process $\nu_\r(\xit^\e)$ to deduce upper bounds for $\g_l$ and $\g_g$. Our key ingredient will be the following lemma.
\begin{Lem}\label{Lem2}
	For any configuration $\e\in \S$ and any $t\geq 0$, we have
	\begin{align}\label{MI}
	\E_w(\nu_\r(\xit^\e))\leq \left(\phi_\g(\r)\right)^t \nu_\r(\e),
	\end{align}
	where $\E_w$ stands for the expectation concerning the probability measure $\p_w$, and
	\begin{align}
	\begin{aligned}
	\phi_\g(\r)&=\exp\left(\lbd-1-\g\right),\quad&&\text{if }\r=1;\\
	&=\exp\left(w_1 \r^{-1}+ w_2 (d-1)\r-1-\g\right),\quad&&\text{if }\r\neq 1.
	\end{aligned}
	\end{align}
\end{Lem}

\begin{Rmk}\label{Rmk1}
	Observe that for any fixed $\g$, $\phi_\g(\r)$ admits a unique global minimum, which is attained at point $\r_*=\sqrt{\nicefrac{w_1}{w_2(d-1)}}$.
\end{Rmk}

\begin{proof}[Proof of Lemma~\ref{Lem2}]
	Recall the definitions~\eqref{Map1} and~\eqref{Map2}. From the definition~\eqref{G}, we have that
	\begin{align}\label{IL}
	\begin{aligned}
	L \nu_\r(\xi)&=-\g \sum_{i\in V} \xi(i)\r^{l(i)}+\sum_{i\in V}\1\{\xi(i)\geq 1\} \left(\sum_{j\in\Nc_i} \wij\r^{l(j)}-\xi(i)\r^{l(i)}\right)\\
	&\leq -(1+\g) \nu_\r(\xi)+\sum_{i\in V}\xi(i)\sum_{j\in\Nc_i} \wij\r^{l(j)}
	\end{aligned}
	\end{align}
	where the inequality follows from the fact that $\1\{\xi(i)\geq 1\} \leq \xi(i)$, for all $i\in V$. Now, by setting
	\begin{align*}
	\begin{aligned}
	c_\r &= \lbd,\quad&&\text{if }\r=1;\\
	&= w_1 \r^{-1}+ w_2(d-1)\r,\quad&&\text{if }\r\neq 1,
	\end{aligned}
	\end{align*}
	note that $\sum_{j\in\Nc_i} \wij\r^{l(j)} \leq c_\r \r^{l(i)}$, for all $i\in V$. Thus,
	\begin{align*}
	\begin{aligned}
	L \nu_\r(\xi)\leq \left(c_\r -1-\g\right)\nu_\r(\xi),
	\end{aligned}
	\end{align*}
	and by classical results on Markovian generators,
	\begin{align*}
	\begin{aligned}
	\frac{d}{dt} \E_w(\nu_\r(\xit^\xi))\leq \left(c_\r-1-\g\right)\E(\nu_\r(\xit^\xi)),
	\end{aligned}
	\end{align*}
	so that~\eqref{MI} is now just a matter of applying Grönwall's lemma in the inequality above.
\end{proof}

From Lemma~\ref{Lem2} we readily have that
\begin{Pro}\label{Prop2}
	\begin{align}\label{IThm}
	\E_w(|\xis|)\leq \exp\left(\lbd-1-\g\right)^t. 
	\end{align}
	In particular, $\xis$ dies out for any $\g\geq \lbd-1$. Therefore, $\g_g\leq \lbd-1$.
\end{Pro}

\begin{proof}
	First, note that~\eqref{IThm} is just inequality~\eqref{MI} set up with $\xi= e_o$ and $\r=1$. Now, let us assume that $\g>\lbd-1$. Since in this case $\E_w(|\xis|)\to 0$ as $t\to \infty$, the classical Markov inequality readily implies
	\begin{align}
	\p_w\left(|\xis|\geq 1,\;\forall t\geq 0\right)=\lim_{t\to \infty} \p_w\left(|\xis|\geq 1\right)=0.
	\end{align}
	Therefore, $\xis$ dies out for any $\g>\lbd-1$, and the upper bound in~\eqref{ub2} is proven. It remains now to show that our process also dies out at point $\g=\lbd-1$. Indeed, since~\eqref{IThm} implies
	\begin{align}\label{c1}
	\lim_{t\to\infty} \E_w\left(|\xis|\right)\leq 1,
	\end{align}
	if $\xis$ survives at $\g=\lbd-1$, then Lemma~\ref{Lem1} in Appendix~\ref{AP1} would imply that (almost surely) $|\xis|>1$ for all large $t$, hence a contradiction. Therefore $\xis$ dies out at $\g=\lbd-1$.
\end{proof}

\begin{Rmk}
	We observe that Proposition~\ref{Prop2} still holds if, instead of a homogeneous tree, the network of neurons is a connected graph of bounded degree. This is so because the argument used to establish inequality~\eqref{MI}, when $\r=1$, requires only that $\lbd<\infty$ and, by definition, this condition holds if the graph representing the network of neurons has a bounded degree.   
\end{Rmk}

Our next step, in proving Theorem~\ref{Thm2}, concerns the critical point $\g_l$. We are going to prove that
\begin{Pro}\label{Prop1}
$\g_l\leq 2\sqrt{w_1 w_2(d-1)}-1$.
\end{Pro}

\begin{proof}
	Suppose $\g>2\sqrt{w_1 w_2(d-1)}-1$, and set
	\begin{align*}
	M_t=\frac{\nu_{\r_*}(\xis)}{\left(\phi_\g(\r_*)\right)^t},
	\end{align*}
	where $\r_*^2=\nicefrac{w_1}{w_2(d-1)}$. Let $\Fc_t=\sigma(\xi_s^{e_o}:s\leq t)$. By the Markov property and inequality~\eqref{MI} we may find that
	\begin{align}\label{eq1}
	\E_w\left(\nu_{\r_*}(\xi_{t+s}^{e_o})\biggr\rvert \Fc_t\right)=\E_w\left(\nu_{\r_*}(\xi_s^{\xi_t^{e_o}})\right)\leq \left(\phi_\g(\r_*)\right)^s \nu_{\r_*}(\xi_t^{e_o}), \quad \forall s,t\geq 0.
	\end{align}
	Using now the fact that $\phi_\g(\r_*)<1$ for any $\g>2\sqrt{w_1 w_2(d-1)}-1$, we readily get from~\eqref{eq1} that $M_t$ is a supermartingale and, since it is non-negative, it converges almost surely (see e.g.,~\cite{Chung2001}, p. 351). This, together with the fact that $\left(\phi_\g(\r_*)\right)^t\to 0$ as $t\to \infty$, implies that $\nu_{\r_*}(\xis)\to 0$ almost surely as $t\to \infty$, and so $\xis$ cannot survive locally. Therefore, $\g_l\leq 2\sqrt{w_1 w_2(d-1)}-1$.
\end{proof}

Having obtained the upper bounds for $\g_l$ and $\g_g$, deducing the lower bounds remains as the final step. We start by showing that $\g_g\geq d-2$. To do so, we consider an auxiliary Markov process $\et=(\et(i), i\in V)$ on $\T_d$ taking values in $\{0,1\}^V$ that is defined as follows. Regarding neurons at state 1 as {\it active}, when they have membrane potential larger than 0; and neurons at state 0 as {\it quiescent}, when they have null membrane potential; the process evolves according to the following rules. Active neurons spike --- they become quiescent and at the same time make all of their neighboring neurons active --- at rate 1, or they become spontaneously quiescent at rate $\g$. Quiescent neurons become active only if, at least one of their active neighboring neurons, if there are any, spikes. We note that this process was first considered by Ferrari et al. in~\cite{Ferrari2018} (see Section 3 therein). In what follows, we shall denote by $\et^\e$ the process $\et$ starting from $\eta_0=\e\in \{0,1\}^V $.

We observe that the process $\et$ can be seen as a version of our model with the same spiking and leakage parameters, but in which we do not allow neurons to accumulate more than one unit of potential. This nice point of view is key in our argument since it leads to a monotone coupling that will be needed in the next proposition: For fixed $w$, using the graphical representation, we can construct $\et$ and $\xit$ simultaneously, but for the process $\et$ we use only the Poisson processes $N_i^*$ and $N_{i,1}^{\dagger}$, in such a way that
\begin{align}\label{d1}
\e(i)\leq \xi(i), \forall i\in V \quad \mbox{implies}\quad \et^\e(i)\leq \xit^\xi(i), \forall i\in V.
\end{align}

With this property, we now can prove that
\begin{Pro}\label{Prop3}
$\g_g \geq d-2$.
\end{Pro}

\begin{proof}
	For fixed $w$,~\eqref{d1} readily implies that
	\begin{align}
	\p\left(|\ets|\geq 1\right) \leq \p_w\left(|\xis|\geq 1\right)
	\end{align}
	for any $t\geq 0$. Hence,
	\begin{align*}
	\p\left(|\ets|\geq 1,\; \forall t\geq 0\right)&=\lim_{t\to \infty} \p\left(|\ets|\geq 1 \right) \\
	&\leq \lim_{t\to \infty} \p_w\left(|\xis|\geq 1\right)=\p_w\left(|\xis|\geq 1,\; \forall t\geq 0\right).
	\end{align*}
	Therefore, it is enough to show that the process $\ets$ survives whenever $\g<d-2$. We do this as follows. Let $\zeta_t$ be a process on $\T_d$ that evolves according the following rules. Start with only one active neuron, the root. It waits a mean 1 exponential time at the end of which it spikes, becoming quiescent and simultaneously making all of its neighbors, except its parent, active. In general, each new active neuron, when it spikes, will make active all of its neighboring neurons, except the parent one. This restriction is to ensure that once a neuron is active, and then becomes quiescent, it will never become active again. Neurons become quiescent if they spike or when the leakage occurs, the latter happening at the rate $\g$, independently.
	
	From the rules above, it is not difficult to see that we can construct $\zeta_t$ and $\ets$ simultaneously, in such a way that $|\zeta_t|\leq |\ets|$ for all $t\geq 0$. Furthermore, since $\T_d$ have no cycles, $|\zeta_t|$, which is just the number of active neurons in $\zeta_t$ at time $t$, defines a continuous-time (non-spatial) branching process with offspring distribution given by $p(0)=\frac{\g}{1+\g}$ and $p(d-1)=\frac{1}{1+\g}$. Now, as it is well-known, branching processes survive if and only if the mean offspring distribution is strictly greater than $1$. Thus, if $\frac{d-1}{1+\g}>1$ or, equivalently, $\g<d-2$, we have that the process $|\zeta_t|$ survives, and so does $|\ets|$. Therefore, $\g_g\geq d-2$.
\end{proof}

\begin{Rmk}
	Given Proposition~\ref{Prop2}, using the coupling argument establishing Proposition~\ref{Prop3}, one may prove the following assertion: For any $d\geq 2$,
	\begin{align}\label{E1}
	\exp\left(d-2-\g\right)^t \leq \E_w(|\xis|) \leq \exp\left(\lbd-1-\g\right)^t. 
	\end{align}
\end{Rmk}

The following concludes the proof of Theorem~\ref{Thm2}.

\begin{Pro}\label{prop4}
For all $d\geq 2$, we have $\g_l>0$.
\end{Pro}

\begin{proof}
	We start by pointing out that, since $\T_d$ contains a copy of $\Z$, given $w$, a simple comparison argument gives that, for fixed $\g$, $\xis$ on $\T_d$ dominates $\xit^{e_0}$ on $\Z$, in the sense that we may construct both processes on the same probability space such that one is larger than the other. This implies that $\g_l(\T_d)\geq \g_l(\Z)$, and so, it is enough to show that $\g_l(\Z)>0$. This follows by showing that $\p\left(\xit^{e_0,\g}(0) \text{ i.o.}\right)>0$, provided $\g$ is small enough.
	
	We apply a general method developed by Bramson and Durrett~\cite{Durret1988}. The idea is proving that for any $\eps>0$ and $\g$ small enough, the process $\xit^{e_0}$ dominates a 1-dependent oriented percolation having open sites with probability $1-\eps$. To get to that, our first step is to introduce the 1-dependent oriented percolation (see~\cite{Durret1984} for details). Let
	\begin{align*}
	\l=\left\{(m,n)\in\Z^2:m+n\text{ is even}\right\},
	\end{align*}
	and for $(m,n)\in\l$, let $U(m,n)$ be Bernoulli random variables (r.v.'s) with parameter $p$ to indicate whether $(m,n)$ is open or not. Make these r.v.'s 1-dependent, that is, if $z_1,\ldots,z_r$ are points of $\l$ with $\|z_k-z_l\|_{\l}>1$ for all $1\leq k\neq l\leq r$, where $\|(m,n)\|_{\l}=\frac 12(|m|+|n|)$, then $U(z_1),\ldots,U(z_r)$ are independent. We say that there is an open path from $(i,0)$ to $(j,n)$ if there is a sequence of points $i=m_0,m_1,\ldots,m_n=j$ such that $|m_k-m_{k+1}|=1$ for $0\leq k< n$, and $(m_k,k)$ is open for $0\leq k\leq n$. Let $C_n^0=\left\{j:\text{there is an open path from $(0,0)$ to $(j,n)$}\right\}$. We think $C_n^0$ as the set of open sites at time $n$ when the root (0 in this case) is open at time 0.
	
	Having defined the percolation process, we now proceed to establish a relationship between $C_n^0$ and our process $\xit^{e_0}$ to compare them. For integers $m$ and $n$, let
	\begin{align}\label{map}
	\begin{aligned}
	&I=[-N,N],\quad I_m=2mN+I\\
	&B=(-4N,4N)\times [0,T],\quad B_{m,n}=(2mN,nT)+B
	\end{aligned}
	\end{align}
	where $N\geq 1$ is an integer, and $T=cN$ for some constant $c>0$ to be specified later. Here $i+S=\{i+j:j\in S\}$. For each $m\in \Z$, we say that an interval $I_m$ is {\it good} if each quiescent neuron within $I_m$ (if there are any) is isolated, that is, its nearest-neighbors, in $I_m$, are active. Let $\xit^I$ be the process starting from a state with $I$ good. We say that $(m,n)\in\l$ is open and set $U(m,n)=1$ if $\zeta_{nT}^{I_m}$, the process $\xit^{I_m}$ starting with $I_m$ good at time $nT$, satisfies that $\zeta_t^{I_m}\times \{t\}\subset B_{m,n}$ for all $t\in \left[nT,(n+1)T\right]$, the intervals $I_{m-1}$ and $I_{m+1}$ are both good at time $(n+1)T$ and $H_{m,n}:=D\cap B_{m,n}=\varnothing$, where $D=\left\{(i,T_{i,k}^{\dg}):k\geq 1,i\in\Z\right\}$ is the set of leaking times in the graphical representation of the process.
	
	It is not difficult to see that the r.v.'s $U(m,n)$ previously defined are 1-dependent. Indeed, this follows immediately from the fact that $B$ was chosen in~\eqref{map} so that $B_{k,l}\cap B_{m,n}=\varnothing$ when $\|(k,l)-(m,n)\|_{\l}>1$. Moreover, if there is an open path from $(0,0)$ to $(m,n)$ in the percolation process, then $I_m$ is good at time $nT$, and all active neurons within $I_m$ are also active in $\xi_{nT}^{I_0}$. As a result, the process $\xit^I$ starting with $I$ good dominates a 1-dependent oriented percolation having open sites with probability $\p\left(U(m,n)=1\right)=p$. If that $p$ is large, known results (see Durrett~\cite{Durret1984}, Section 10) show that $(0,0)$ is in an infinite open cluster containing infinitely many vertices such as $(0,2n)$. That implies local survival for the process $\xit^I$, by the coupling.
	
	We now argue that for any $\eps >0$ we can find $N$ large enough and then $\g>0$ small enough so that
	\begin{align}\label{j0}
	\p\left(U(m,n)=1\right)\geq 1-\eps,
	\end{align}
	for any $(m,n)\in \l$. It suffices to consider $(m,n)=(0,0)$. Let $\xit^I$ be the process starting from a state with $I=[-N,N]$ good and let $\eps>0$ be arbitrary. We write
	\begin{align}\label{j1}
	\p\left(U(0,0)=0\right)\leq \p\left(H_{0,0}\neq \varnothing\right)+ \p\left(U(0,0)=0|H_{0,0}=\varnothing\right).
	\end{align}
	Let $R_t$ and $L_t$ be the rightmost and leftmost active neuron at time $t$ respectively. Conditioned on the event $\left\{H_{0,0}=\varnothing\right\}$, we have $\left\{i\in \Z:\xit^I(i)\geq 1\right\}\subset[L_t,R_t]$ for any $t\leq T$, and since there are no leaking marks in that period, the interval $[L_t,R_t]$ is good at any time $t\in [0,T]$. Moreover, $R_t$ jumps to $R_t+1$ (resp. $L_t$ jumps to $L_t-1$) at rate $1$, so that, from the strong law of large numbers, $\nicefrac{R_t}{t} \to 1$ (resp. $\nicefrac{L_t}{t} \to 1$) almost surely. Therefore, if we set $c=\nicefrac 52$ and pick $N$ large enough, $[L_T,R_T]$ will contain the intervals $I_{-1}$ and $I_1$, which must be both good, with a probability at least $1-\nicefrac{\eps}{2}$. Hence,
	\begin{align}\label{j3}
	\p\left(U(0,0)=0|H_{0,0}=\varnothing\right)\leq \frac{\eps}{2}.
	\end{align}
	It remains to control the other term on the right-hand side of~\eqref{j1}. For that sake, first, note that, for the $N$ previously picked, the total number $M$ of leaks that may occur inside $B$ is bounded above by a Poisson r.v. of parameter $\g(8N+1)T$, since the times at which neurons of $B$ spike are modeled by independent rate-$\g$ Poisson processes. Therefore,
	\begin{align}\label{j2}
	\p\left(H_{0,0}\neq \varnothing\right)\leq \E(M)\leq  2\g N(8N+1)\leq \frac{\eps}{2},
	\end{align}
	if $\g>0$ is small enough.~\eqref{j0} is now just a matter of replacing~\eqref{j3} and~\eqref{j2} in~\eqref{j1}. Therefore, the process $\xit^I$, starting with $I=[-N,N]$ good, survives locally if $\g>0$ is small. To extend this conclusion for the process $\xit^{e_0}$, we note that, since $N\geq 1$ is finite, there is a positive probability that $I$ is good by time 1. The assertion then follows by the Markov property.	
\end{proof}

\paragraph{Application.}
We conclude this section with an application of some of the results above to a particular version (on $\T_d$) of the system of spiking neurons studied by Ferrari et al.~\cite{Ferrari2018}. That particular version is essentially equivalent to the auxiliary Markov process $\et$ introduced above. Consequently, they share the same critical points for $\g$. We are going to show that $\et$ also has two phase transitions when started from $e_o$. Let $\g_1$ and $\g_2$ be the two critical values of $\ets$:
\begin{align*}
\begin{aligned}
\g_1&=\g_1(\T_d)=\sup\left\{\g: \p\left( \et^{e_o,\g}(o)\geq 1 \mbox{ i.o.}\right)>0\right\},\\
\g_2&=\g_2(\T_d)=\sup\left\{\g: \p\left(|\et^{e_o,\g}|\geq 1,\; \forall t\geq 0\right)>0\right\}.
\end{aligned}
\end{align*}
Consider the process $\xis$ having matrix $w=(\wij)_{i,j\in V}$ with entries given by $\wij=\1\{j\in\Nc_i\}$, for each $i\in V$. From~\eqref{d1} and Propositions~\ref{Prop2} and~\ref{Prop1}, we get that $\g_1 \leq \g_l\leq 2\sqrt{d-1}-1$ and $\g_2 \leq \g_g\leq d-1$. This, together with the bound $\g_2 \geq d-2$ (which follows from the coupling argument proving Proposition~\ref{Prop3}), shows that $\g_1<\g_2$ (the process $\ets$ has two phase transitions) for all $d\geq 6$, and so does the particular system (on $\T_d$ and with initial configuration $e_o$) considered by~\cite{Ferrari2018}.

\appendix
\section{Appendix}\label{AP1}
\begin{Lem}\label{Lem5}
	If $\g' < \g''$, then for any $t\geq 0$ we have
	\begin{align}
	\p_w\left(|\xit^{e_o,\g'}|\geq 1\right) \geq \p_w\left(|\xit^{e_o,\g''}|\geq 1\right).
	\end{align}
\end{Lem}

\begin{proof}
	It is enough to construct the processes $\xit^{e_o,\g'}$ and $\xit^{e_o,\g''}$ simultaneously in such a way that 
	\begin{align}\label{eqP1}
	\xit^{e_o,\g''}(i)\leq \xit^{e_o,\g'}(i), \quad \forall i\in V, \forall t\geq 0.
	\end{align}
	To achieve this, we proceed as follows. First, we define both processes using the same Poisson processes $N_i^*$ of rate 1. Then, we consider independent Poisson processes $N_{i,k}^{\dagger,\g''}$ of rate $\g''$ for $k\geq 1$ to give the leaking times for both processes, and we introduce an ``acceptance probability'' equal to $\nicefrac{\g'}{\g''}$, to obtain the right leakage rate for $\xit^{e_o,\g'}$: with probability $\nicefrac{\g'}{\g''}$ an event time of $N_{i,k}^{\dagger,\g''}$ is accepted, and the leakage that occurs in $\xit^{e_o,\g''}$ may also occur in $\xit^{e_o,\g'}$.
\end{proof}

\begin{Lem}\label{Lem1}
	On $\O_{\infty}:=\left\{|\xis|\geq 1, \,\forall t\geq 0\right\}$, we have that $|\xis|\to \infty$ almost surely.
\end{Lem}

\begin{proof}
	The result is proven by arguing as Lemma 2.1. in~\cite{Bramson1991} using, however, that, if $|\xis|\leq N$ for some integer $N\geq 1$, then the probability that all membrane potential within the system will vanish before any spiking is at least $\left(\nicefrac{\g}{1+\g}\right)^N$.
\end{proof}

\section*{Acknowledgements}
This article is part of the postdoctoral project (CNPq grant 155972/2018-9) of the author at IME-USP and has been produced as part of the activities of FAPESP Research, Innovation and Dissemination Center for Neuromathematics (grant 2013/07699-0, S. Paulo Research Foundation). The author would like to thank Antonio Galves and Aline Duarte for stimulating discussions about this subject.

\bibliographystyle{apacite}

\end{document}